\documentclass[a4paper,12pt]{article}
\usepackage{amsfonts}
\usepackage{latexsym}
\usepackage{amsthm}
\usepackage{amssymb}
\usepackage{amsmath}
\usepackage{eufrak}
\usepackage{mathrsfs}
\usepackage{geometry}
\usepackage{graphics}
\usepackage{comment}
\usepackage[all]{xy}
\usepackage{rotating}
\usepackage{algorithmic}
\usepackage{algorithm}
\usepackage{lscape}
\usepackage{color}

\newtheorem{theo}{Theorem}[section]

\newtheorem{Lemma}[theo]{Lemma}
\newtheorem{FundLem}[theo]{Fundamental Lemma}

\newtheorem{Prop}[theo]{Proposition}

\newtheorem{Rm}[theo]{Remark}

\newtheorem{Ex}[theo]{Example}

\newtheorem*{tma}{Theorem}
\newtheorem*{tma1}{Theorem \ref{maior}}
\newtheorem*{notation}{Notation}
\newtheorem{pps0}[theo]{Proposition}
\newcommand{\lra}{ \longrightarrow }

\renewcommand{\Im}{{\rm Im}}

\newcommand{\Sym}{{\rm Sym}}
\newcommand{\Fix}{{\rm Fix}}
\newcommand{\Supp}{{\rm Supp}}

\author{Natalia A. Viana Bedoya,\  Daciberg Lima Gon\c calves and\\ Elena A. Kudryavtseva }

\title{Indecomposable branched coverings over the projective plane by surfaces $M$ with $\chi(M)\leq 0$}
\begin{document}
\maketitle

\begin{abstract} In this work we study the decomposability property of  branched \hspace*{-3pt} coverings 
of degree $d$ odd,
over the projective plane, where the covering surface has Euler characteristic $\leq 0$. The latter condition is equivalent to
say that the  defect of the covering is greater than  $d$.
  We  show that, given a   datum $\mathscr{D}=\{D_{1},\dots,D_{s}\}$ with an even defect  greater than  $d$, it is realizable by
  an indecomposable branched covering over the projective plane.
  The case when $d$ is even is known.

\end{abstract}

\section*{Introduction}

 A branched covering $\phi: M \lra N$ of degree $d$ between closed  surfaces determines
a finite collection  $\mathscr{D}$ of  partitions of $d$, {the branch
 datum of degree $d$}, in $1-1$ correspondence  with the {branch point set} $B_{\phi} \subset N$. The {\it total defect} of $\mathscr{D}$ is defined by
 $\nu(\mathscr{D})=\sum_{x \in B_{\phi}}(d-\# \phi^{-1}(x))$.
 Given a collection of partitions $\mathscr{D}$ of $d$ and $N \neq S^{2},\mathbb RP^2$, there are  necessary and sufficient conditions on $\mathscr{D}$ 
 to realize it  as branch datum of a branched
 covering of degree $d$ over $N$ with a connected covering surface, for more details see \cite{EKS}, \cite{Ez}, \cite{Hu}.
Such collections  are called either
{\it admissible data} if $\chi(N)\le0$, or {\it nonorientable-admissible} if $N=\mathbb{R}P^{2}$ and the covering surface is nonorientable.
Due to \cite {EKS}, these conditions are either
$\nu(\mathscr{D})\equiv0\pmod2$ if $\chi(N)\le0$, or
$
d-1\le\nu(\mathscr{D})\equiv0\pmod2
$
if $N=\mathbb{R}P^{2}$ and the covering surface is nonorientable,
 see (\ref {hcpp}) for $N=\mathbb{R}P^{2}$.

 Decomposability properties  of branched coverings between surfaces provide
 three classes of admissible data: those data that are realizable only by decomposable branched coverings, those that are realizable only by
 indecomposable branched coverings and those that are realizable by both, decomposable and indecomposable branched coverings. A
 datum realizable by
 a decomposable primitive  branched covering over $N$ with a connected covering surface
  is  called  {\it decomposable primitive datum on $N$}.
 A characterization  of the
 admissible data which are realizable by decomposable primitive branched coverings over $N\ne S^2$
 is known. This characterization
follows from  \cite[Proposition 2.6]{BN&GDL1}  for $\chi(N)\leq 0$, and \cite{BN0} for $N=\mathbb{R}P^{2}$.
The proof for $\chi(N)=1$, which is in \cite{BN0},
is similar to the proof of   Proposition 2.6 in   \cite{BN&GDL1}.

\begin{pps0}[Proposition 2.6 \cite{BN&GDL1}]\label{geral}
Admissible data $\mathscr{D}$ are decomposable primitive on $N, $ with $\chi(N)\leq 0$,
 if and only if there exists a factorization of $\mathscr{D}$
  such that its left factor is a non-trivial  admissible datum. \qed
\end{pps0}

 The question of realization by indecomposable branched coverings of an admissible   datum is interesting
for decomposable data, otherwise clearly the problem has a positive  solution.
This question  has been completely solved in  \cite{BN&GDL1} for the case where $N$ is  a closed surface with $\chi(N)\leq 0$. Namely:

  \begin{tma}[Theorem 3.3 \cite{BN&GDL1}]
Every non-trivial  admissible data are realized on any $N$, with $\chi(N)\leq 0$, by an
indecomposable (and hence primitive) branched covering. \qed
\end{tma}

In this case, there does not exist an admissible datum realizable only by decomposable branched coverings, therefore decomposable and
indecomposable realizations coexist (whenever a decomposable realization exists) for the same data.

 It remains  to study  the  ``indecomposability''  of branched coverings between surfaces
where  $\chi(N)=1$, i.e. $N=\mathbb{R}P^{2}$, since we are not considering the case $N=S^2$.  So let $M \to \mathbb{R}P^{2}$ be a branched covering of degree $d$.
The even degree case, i.e. when $d$ is even, has been solved in \cite{BN&GDL2}. The result is:

  \begin{tma}[Theorem 3.7 \cite{BN&GDL2}]\label{even}
 Let $d$ be even and $\mathscr{D}=\{D_{1},\dots,D_{s}\}$ a nonorientable-admissible datum (see (\ref {hcpp}) or theorem \ref {realization})
 such that $s>0$ and $D_i\ne[1,\dots,1]$ for any $i\in \{1,\dots,s\}$. Then $\mathscr{D}$ is realizable by an indecomposable
   branched covering over $\mathbb{R}P^{2}$ if and only if
  at least one of the following conditions holds:\\
(1) $d=2$, or\\
 (2) there is  $i \in \{1,\dots,s\}$ such that
   $D_{i} \neq [2,\dots,2]$, or\\
(3) $d>4$ and $s>2$.\qed
 \end{tma}

 Except for the case $d=2$, where the branched covering clearly is never decomposable, the  theorem
above together    with the analogue for $N=\mathbb{R}P^{2}$ of Proposition 2.6 \cite{BN&GDL1}   characterize nonorientable-admissible data realizable by both,
decomposable primitive and indecomposable branched coverings of even degree over $\mathbb{R}P^{2}$. Moreover nonorientable-admissible data like either
$\{[2,\dots,2],[2,\dots,2]\}$ or $\{[2,2],[2,2],\dots,[2,2]\}$ with $d>2$ are
realizable only by decomposable branched coverings.

In this work we study the case of odd degree with
$N=\mathbb{R}P^{2}$ and $\nu(\mathscr{D})>d-1$ (i.e.\ $\chi(M)\le0$, compare (\ref {hcpp})). Our main result is:

\begin{tma1}
Let  $\mathscr{D}$ be a collection of partitions of an odd integer $d$
such that $d-1<\nu(\mathscr{D})\equiv0\pmod2$ (compare (\ref {hcpp})).
Then it can be realized as the branch datum of an indecomposable (and hence primitive) branched covering of degree $d$ over the projective plane with a connected covering surface.
\end{tma1}

Our technique of proving this theorem will allow us to show that certain collections of partitions of an odd integer $d$ are realizable as branch data of
branched coverings over the 2-sphere with a connected covering surface (see Theorem \ref {maior:sphere}).

\section{Preliminaries, terminology and notation}
  \subsection{Permutation groups} \label {subsec:permut}
We denote by $\Sigma_{d}=\Sym_\Omega$ the symmetric group on a set $\Omega$ with $d$ elements and
 by $1_{d}$ its identity element. If $\alpha \in \Sigma_{d}$ and $x \in
 \Omega$, $x^{\alpha}$ is the image of $x$ by $\alpha$. An explicit permutation $\alpha$ will be written sometimes as a product of disjoint cycles, i.e. its \emph{cyclic decomposition}.
The set of lengths of the
 cycles in the cyclic decomposition of $\alpha$, including the trivial ones, defines a partition of $d$, say
 $D_{\alpha}=[d_{\alpha_{1}},\dots,d_{\alpha_{t}}]$,  called
 \emph{the cyclic structure of} $\alpha$. Define
 $\nu(\alpha):=\sum_{i=1}^{t}(d_{\alpha_{i}}-1)$, then $\alpha$ will be an \emph{even permutation} if $\nu(\alpha) \equiv 0 \pmod{2}$.
Given  a partition $D$ of $d$, we say  $\alpha \in
D$ if the cyclic structure of $\alpha$ is
$D$ and  we put $\nu(D):=\nu(\alpha)$.

For $1<r \leq d$, a permutation $\alpha \in \Sigma_{d}$ is called a
$r$-\emph{cycle} if, in its cyclic decomposition, its unique non-trivial cycle
has
length $r$. Permutations $\alpha,\beta \in \Sigma_{d}$ are \emph{conjugate}
if there is $\lambda \in \Sigma_{d}$  such that
$\alpha^{\lambda}:=\lambda \alpha \lambda^{-1}=\beta$. It is a known fact that
conjugate permutations have the same cyclic structure.

  Given a permutation group $G$ on $\Omega$  and $x \in \Omega$,  one  defines the {\it isotropy subgroup of $x$},
 $G_{x}:=\{g \in G: x^{g}=x\}$, and the {\it orbit of $x$ by $G$}, $x^{G}:=\{x^{g}:g\in G\}$. For $H \subset G$, the subsets $\Supp(H):=\{x
 \in \Omega: x^{h} \neq x \textrm{ for some $h \in H$}\}$ and
 $\Fix(H):=\{x \in \Omega:x^{h}=x \textrm{ for all $h \in H$}\}$ are defined. For
 $\Lambda \subset \Omega$ and $g \in G$,
 $\Lambda^{g}:=\{y^{g}:y \in \Lambda\}$.

  $G$ is said
 \emph{transitive} if for all
 $x,y \in \Omega$ there is $g \in G$ such that
 $x^{g}= y$.
 A non-empty subset $\Lambda \subset \Omega$ is a \emph{block} of a transitive $G$ if
 for each $g \in G$ either $\Lambda^{g} = \Lambda$ or
 $\Lambda^{g} \cap \Lambda =\emptyset$. A block $\Lambda$ is
 \emph{trivial} if either $\Lambda =\Omega$ or $\Lambda=\{ x \}$ for some $x \in
 \Omega$. Given a block $\Lambda$ of $G$, the set
 $\Gamma:=\{\Lambda^{\alpha}:\alpha \in G\}$ defines a partition of $\Omega$
 into blocks. This set is called \emph{a system of blocks
 containing} $\Lambda$ and the cardinality of $\Lambda$ divides the cardinality of $\Omega$. $G$ acts naturally on $\Gamma$.
 A transitive permutation group is
 \emph{primitive}
 if it determines only  trivial blocks. Otherwise it is
 \emph{imprimitive}.

\begin{Prop}[Corollary 1.5A, \cite{DM}]\label{dixon}
Let $G$ be a transitive permutation group on a set $\Omega$ with at least two
points. Then $G$ is primitive if and only if each isotropy subgroup $G_{x}$,
for $x \in \Omega$,
is a maximal subgroup of $G$. \qed
\end{Prop}

It will be important for us to recognize when a  permutation $\alpha$ already provides primitivity for any subgroup that contains $\alpha$.

\begin{Ex}\label{referee}
If gcd$(l,d)=1$ and $l$ is greater than any non-trivial divisor of $d$ then any transitive permutation group $G < \Sigma_{d}$ containing a $l$-cycle is primitive
(this holds, for example, if $d=2l\pm1$). In fact, we can assume that $G$ contains the cycle $(1 \dots l)$. If there is a block of $G$ containing two elements $i$ and $j$
with $i \leq l$ and $j>l$ then it also contains $1,\dots,l$, thus the cardinality of the block is $\geq l+1$, hence it  equals $d$ and the block is trivial.
 Otherwise the cardinality of each block divides both $l$ and $d-l$, hence it equals 1, thus all blocks of $G$ are trivial. Hence $G$ is a primitive permutation group.
\end{Ex}

\subsection{Branched coverings over the projective plane}\label{pbc}
A surjective continuous open map $\phi:M \lra N$ between closed surfaces such
that:
\begin{itemize}
\item for $x \in N$, $\phi^{-1}(x)$ is a finite
set, and
\item there is  a
 discrete set $B_{\phi} \subset N$  such that the
restriction $\hat{\phi}:=\phi|_{M\setminus\phi^{-1}(B_{\phi})}$ is an ordinary
unbranched
covering of degree $d$,
\end{itemize}
is called a \emph{branched covering of degree d over N}
 and it is  denoted  by
$(M,\phi,N,B_{\phi},d)$. The surface $N$ is the \emph{base surface}, $M$ is
the \emph{covering surface} and $B_{\phi}$ is the \emph{branch point set}.
If $B_{\phi}=\varnothing$ then we also call $\phi$ an {\em unbranched covering}, and if $B_{\phi}\ne\varnothing$ then we also call $\phi$ a {\em proper branched covering}.
 Its \emph{associated unbranched covering} is denoted by
$(\widehat{M},\hat{\phi},\widehat{N},d)$, where
$\widehat{N}:=N\setminus B_{\phi}$ and $\widehat{M}:=M\setminus\phi^{-1}(B_{\phi})$. It is known that
$\chi(\widehat{M})=d\chi(\widehat{N})$, equivalently
\begin{eqnarray}\label{cE}
 \chi(M)-\#\phi^{-1}(B_{\phi})=d(\chi(N)-\#B_ {\phi}).
\end{eqnarray}
The set $B_{\phi}$ contains
the image of the set of the points in $M$ in that $\phi$ fails to be a local
homeomorphism. Then each $x \in B_{\phi}$ determines a
partition $D_{x}$ of $d$ (possibly $D_{x}=[1,\dots,1]$), defined by the local degrees of $\phi$ on each component
in the preimage of a small disk $U_{x}$ around  $x$, with
$U_{x}\cap B_{\phi}=\{x\}$. The collection $\mathscr{D}:=\{D_{x}\}_{x \in B_{\phi}}$ is called
the \emph{branch datum} and its  \emph{total defect} is the
non-negative integer
defined by
$\nu({\mathscr{D}}):=\sum_{x \in B_{\phi}} \nu(D_{x})$. The total defect  satisfies
the \emph{Riemann-Hurwitz formula} (see \cite{Hurwitz} or \cite{EKS}):
\begin{eqnarray}\label{rhf}
\nu(\mathscr{D})=d \chi(N)-\chi(M)\equiv0\pmod2.
\end{eqnarray}
Associated to $(M,\phi,N,B_{\phi},d)$ we have a permutation group, {\it the monodromy group of $\phi$}, denoted by $G(\phi)$, which is the
image of the \emph{Hurwitz representation}
\begin{eqnarray}\label{Hr}
 \rho_{\phi}: \pi_{1}(N\setminus B_{\phi},z) \lra \Sigma_{d},
 \end{eqnarray}
that sends each class $\alpha \in \pi_{1}(N\setminus B_{\phi},z)$ to a permutation of $\phi^{-1}(z)=\{z_{1},
\dots,z_{d}\}$, which indicates the terminal point of the lifting of a loop
in $\alpha$ after fixing the initial point \cite {Hurwitz}. In particular,  for $x \in B_{\phi}$, let $c_{x}$
be a path from $z$ to a small circle $a_{x}$ about $x$ and define the loop class $\mathbf{u}_{x}:=[c_{x}a_{x}c_{x}^{-1}]$. Then the cyclic structure
(see \S \ref {subsec:permut}) of the permutation $\alpha_{x}:=\rho_{\phi}(\mathbf{u}_{x})$ is given by $D_{x}$ and
$\nu(\prod_{x \in B_{\phi}} \alpha_{x}) \equiv \nu(\mathscr{D}) \pmod{2}$.

 \begin{theo}[See \cite{Hu} and \cite{Ez}] \label {thm:1.2}
 Let $N$ be a closed connected surface,  $\mathscr{D}=\{D_1,\dots,D_s\}$ a finite collection of partitions of $d$ and $F=\{x_1,\dots,x_s\}\subset N$ such that $\# F=s
= \#\mathscr{D}$. If it is possible to define a representation $\rho:\pi_{1}(N\setminus F,z) \lra \Sigma_{d}$
with $\rho(\mathbf{u}_{x_i})\in D_i$, $1\le i\le s$ (and with a transitive image Im$\rho=\rho(\pi_1(N\setminus F,z))<\Sigma_d$), then $\mathscr{D}$ is realizable on $N$,
i.e.\ it is the branch datum of a branched covering on $N$ (resp.\ with a connected covering surface $M$) having $\rho$ as its Hurwitz representation.\qed
\end{theo}

\begin{Rm}\label{mon}
If $N=\mathbb{R}P^{2}$ and $\mathscr{D}=\{D_{1},\dots,D_{s}\}$,
in order to define $\rho_{\phi}$, we need at least permutations $\alpha_{i} \in D_{i}$, for $i=1,\dots,s$, such that $\prod_{i=1}^{s}\alpha_{i}$ is a square, as result of the presentation of
$\pi_{1}(\mathbb{R}P^{2}\setminus\{x_{1},\dots,x_{s}\})=\langle a, \mathbf{u}_{1},\dots,\mathbf{u}_{s} \vert  \prod_{i=1}^{s} \mathbf{u_{i}}=a^{-2}\rangle$.
We will also need the transitivity of the subgroup generated by these permutations, in order to obtain a connected covering surface $M$.
\end{Rm}

\begin{Ex}\label{square}
If $r>0$ is an odd natural number then every $r-$cycle is the square of a permutation:
if $\alpha=(a_{1}\;a_{2}\dots a_{r})$ then $\alpha=\beta^{2}$ where
\begin{eqnarray*}
\beta=(a_{1}\;a_{\frac{r+1}{2}+1}\; a_{2}\; a_{\frac{r+1}{2}+2}\dots a_{r}\; a_{\frac{r+1}{2}}).
\end{eqnarray*}
\end{Ex}

We state below  the  main  theorem from \cite {EKS} about the realizability of branched covering  over $\mathbb{R}P^{2}$ in  a slightly
different form, which is more suitable for our purpose.

\begin{theo}[See \cite{EKS}]\label{realization} \label {thm:1.3}
Let $\mathscr{D}$ be a collection of partitions of $d$. Then there is a branched covering $\phi:M\rightarrow \mathbb{R}P^{2}$ of degree $d$, with M connected and nonorientable and with branch datum $\mathscr{D}$ if and only if
\begin{eqnarray}\label{hcpp}
d-1\leq \nu(\mathscr{D})\equiv 0 \pmod{2}.
\end{eqnarray}
Moreover, if $\phi:M\rightarrow \mathbb{R}P^{2}$ is a branched covering and
$\phi_{\#}:\pi_1(M) \rightarrow \pi_1(\mathbb{R}P^{2})$ is trivial then
 $M$ is orientable. If $\phi:M\rightarrow \mathbb{R}P^{2}$ is a branched covering and
$\phi_{\#}:\pi_1(M) \rightarrow \pi_1(\mathbb{R}P^{2})$ is surjective  then
 $M$ is  nonorientable.
\qed
\end{theo}

In the cases   where the covering   space is orientable, the    problem does not appear as a problem over  $\mathbb{R}P^{2}$ but
will lead naturally to a similar question for a branched covering over $S^2$.

 Finally we conclude this section by observing that from the   \emph{Riemann-Hurwitz formula} above we  have
 $\nu(\mathscr{D})=d-1$  if and only if  $\chi(M)=1$ which is equivalent to say that $M$ is homeomorphic to $\mathbb{R}P^{2}$. Also if  $\nu(\mathscr{D})=d$ by the Hurwitz condition (\ref {rhf}) follows that $d$ is even. Therefore if $d$ is odd we can not have    $\nu(\mathscr{D})=d$.

\section{Decomposability}
Given a branched covering, it is
\emph{decomposable} if it can be written as a composition of two non-trivial
branched coverings (i.e.\ both with degree bigger than 1), otherwise it is called
\emph{indecomposable}. In a decomposition of a proper
branched covering, at least one of its factors is a proper branched covering.
 Moreover,
since the degree of a decomposable covering is the product of the degrees of
its factors (see \cite{BBZ}, theorem 2.3), we are interested in branched
coverings with non-prime degree.

\begin{Prop}[Proposition 2.8 \cite{BN&GDL1}]\label{yo}
The covering surface $M$ of a branched covering is connected if and only if its monodromy group is transitive.
A
branched covering with a connected covering surface is decomposable if and only if its monodromy
group is imprimitive.
\qed
\end{Prop}

\begin{Prop}\label{1pto}
A branched covering
$(\mathbb{R}P^{2},\phi,\mathbb{R}P^{2},\{x\},d)$ is decomposable if and only if
  $d$ is not a prime.
\end{Prop}

\begin{proof}
From (\ref{cE}) and (\ref{rhf})
the  total defect of the branch datum is $d-1$ and by (\ref{hcpp}), $d$
 is odd. So the branched covering
$(\mathbb{R}P^{2},\phi,\mathbb{R}P^{2},\{x\},d)$ has branch datum
 $\mathscr{D}=\{[d]\}$.
In a representation
\begin{eqnarray*}
\rho: \pi_{1}(\mathbb{R}P^{2}\setminus\{x\})= \langle
 a,\mathbf{u}_{x}| a^{2}\mathbf{u}_{x}=1 \rangle & \lra & \Sigma_{d}\\
 a & \longmapsto & \alpha,\\
\mathbf{u}_{x}& \longmapsto & \gamma,
\end{eqnarray*}
where
 $\alpha^{2}=\gamma^{-1}$ is a $d$-cycle, necessarily $\alpha$ is a $d$-cycle. Therefore,
 in the group
$G:=\Im\rho=\langle \gamma, \alpha \rangle=\langle \alpha \rangle$,  every isotropy
 subgroup is trivial and, if $d>1$ is not a prime,
 it is contained in a
 proper subgroup of $G$. Then by Proposition
 \ref{dixon}, $G$ is imprimitive, and by Proposition \ref {yo}, the branched covering is decomposable.

The inverse implication is obvious.
\end{proof}

\begin{Lemma}\label{cor}
Let $\alpha \in \Sigma_{d}$ be an even permutation such that $\nu(\alpha)<d-1$ and either
$\Fix(\alpha)\neq \emptyset$ or $\alpha^{2}\neq 1_{d}$.
Then there exists a $d-$cycle $\beta \in \Sigma_{d}$ such that $\alpha \beta$ is also a $d-$cycle and $H=\langle \alpha,\beta \rangle$ is a primitive permutation group.
\end{Lemma}
\begin{proof}
Case (3) of Theorem 3.2 in \cite{BN&GDL1}.
\end{proof}

\begin{Prop}\label{d impar 1}
Let $d>1$ be odd
and $\mathscr{D}=\{D_{1},\dots,D_{s}\}$ a nonorientable-admissible datum of degree $d$ (see (\ref {hcpp}) or Theorem \ref {realization}).
If there is $i\in\{1,\dots,s\}$ such that $D_{i}=[d]$ and $D_j\ne[1,\dots,1]$ for all $j\ne i$, then $\mathscr{D}$
 is realizable by an indecomposable (and hence primitive) branched covering over $\mathbb{R}P^{2}$ if and only if $d$ is prime or $s>1$.
\end{Prop}

\begin{proof}
Suppose $\mathscr{D}=\{D_{1},\dots,D_{s}\}$, $s>1$ and  without loss of generality suppose $D_{s}=[d]$. For
$i=1,\dots,s-1$ choose $\gamma_{i} \in \Sigma_{d}$ with cyclic structure  $D_{i}$.

If $\prod_{i=1}^{s-1} \gamma_{i} \neq 1_{d}$, then its cyclic structure
determines a new partition  $D=[d_{1},\dots,d_{q}]$ of $d$ such that $\nu(D)=d-q
\equiv
\sum_{i=1}^{s-1} \nu(D_{i})=\nu(\mathscr{D})-\nu(D_{s})\equiv d-1 \equiv 0
\pmod{2}$, then $q$ is odd. If $q=1$, define
$\gamma_{s}:=(\prod_{i=1}^{s-1} \gamma_{i})^{-1}$, $\alpha:=(1 \;
1^{\gamma_{s}})$ and the representation
\begin{eqnarray*}
\rho: \langle a,\{\mathbf{u}_{i}\}_{i=1}^{s}\mid a^{2} \prod_{i=1}^{s}
\mathbf{u}_{i} =1 \rangle & \lra &
\Sigma_{d}\\
a & \longmapsto & \alpha, \\
\mathbf{u}_{i} & \longmapsto & \gamma_{i}.
\end{eqnarray*}
Notice that Im$\rho$ is a primitive permutation group, because by the structure of
 $\alpha$, every block containing the element  $1$ contains also
 $1^{\alpha}=1^{\gamma_{s}}$ and since $\gamma_{s}$ is a  $d$-cycle, this
block contains everything, therefore the block is trivial.

 If $q>1$, we apply Lemma \ref{cor} for $\prod_{i=1}^{s-1} \gamma_{i}$ and
 therefore there is a $d$-cycle $\gamma_{s}$ such that $\prod_{i=1}^{s}
 \gamma_{i}$ is a $d$-cycle and $\langle \prod_{i=1}^{s-1} \gamma_{i},
 \gamma_{s}\rangle$ is primitive. Moreover, since $\prod_{i=1}^{s} \gamma_{i}$
 is an odd length cycle, by Example \ref{square} there is $\alpha \in \Sigma_{d}$ such that
 $\alpha^{2}=\prod_{i=1}^{s}
 \gamma_{i}$. We define the following representation:
\begin{eqnarray*}
\rho: \langle a,\{\mathbf{u}_{i}\}_{i=1}^{s}\mid a^{2} \prod_{i=1}^{s} \mathbf{u}_{i} =1 \rangle & \lra &
\Sigma_{d}\\
a & \longmapsto & \alpha^{-1}, \\
\mathbf{u}_{i} & \longmapsto & \gamma_{i},
\end{eqnarray*}
with Im$\rho$ primitive because it contains $\langle \prod_{i=1}^{s-1}
 \gamma_{i},
 \gamma_{s} \rangle$. Since Im$\rho$ is  primitive in both cases above, Proposition \ref{yo} guarantees that
 branched coverings associated by virtue of Theorem \ref {thm:1.2} to each one of the representations
 above are indecomposable.

If $\prod_{i=1}^{s-1} \gamma_{i}= 1_{d}$ and there is some
 $\gamma_{i}$ with a cycle with length $\geq 3$, we change $\gamma_{i}$ by
$\gamma_{i}^{-1}$. If each $\gamma_{i}$ is a product of independent cycles of length less
 than or equal to  2 we replace the symbol of a transposition (which exists because $D_j\ne[1,\dots,1]$ for some $j=1,\dots,s-1$) by a symbol in
 another cycle. Thus, without changing the  cyclic structure of the
 permutations, the new product $\prod_{i=1}^{s}\gamma_{i}$ is different from the
 identity and we are in the case before.

 The implications follow immediately from Proposition \ref{1pto} and Theorem \ref {thm:1.3}.
\end{proof}

\begin{Rm} \label {rem:2.5}
Observe that, among the nonorientable-admissible data studied in the previous proposition, the ones that are realized by indecomposable branched coverings
over $\mathbb{R}P^{2}$ are such that $\nu(\mathscr{D})>d-1$ or $d$ is prime. We want to know if this property  is enough to guarantee the existence of an 
indecomposable realization over $\mathbb{R}P^{2}$ for any nonorientable-admissible branch data of odd degree $d$.  For that  it remains to analyze the cases where the
partitions in $\mathscr{D}$ are all different from $[d]$ and $d$ is odd and non-prime.
\end{Rm}

\section{The case of 2 branch points}

In this section we study the problem in the special case where the number of branch points is
two (which is the minimal possible value, provided that $\nu(\mathscr D)\ge d$). In the next section  we will show that the general
case can be reduced to this case. The main result of the section is:

  \begin{theo}\label{maior}
If $\mathscr{D}=\{D_{1},D_{2}\}$ is such that $d$ is odd and $d< \nu(\mathscr{D})\equiv 0 \pmod{2}$ then $\mathscr{D}$ is realizable by an indecomposable branched covering over
$\mathbb{R}P^{2}$ with a connected covering surface.
 \end{theo}

 To show the about result the main tool is the Lemma \ref {fundlem2}, which can be useful for other aplications.  Several of the arguments used in the proof of this Lemma are 
 similar to some arguments which appear in the proof of the 
 lemma below:

\begin{Lemma}[Corollary 4.4 \cite{EKS}]\label{algEKS}
Let $\overline{D}_{1},\overline{ D}_{2}$ be partitions of $\overline{d}\in \mathbb{N}$  such that $\nu(\overline{D}_{1})+\nu(\overline{D}_{2})\geq \overline{d}-1$ and
$\nu(\overline{D}_{1})+\nu(\overline{D}_{2})\equiv \overline{d}+1 \pmod{2}$.
Given $\overline{\lambda} \in \overline{D}_{1}$ there exists $ \overline{\beta}\in \overline{D}_{2}$ such that $\overline{\lambda}\,\overline{\beta}$ is a $\overline{d}$-cycle
(and hence $\langle \overline{\lambda}, \overline{\beta}\rangle$ acts transitively on $\{1,\dots,\overline{d}\}$). \qed
\end{Lemma}

 Althought the proof of Lemma \ref {fundlem2}
 is more elaborate. The above Lemma is not sufficient to get our Lemma.

It is worth to say that  due to Proposition \ref {d impar 1} and Remark \ref {rem:2.5}, $\mathscr{D}=\{D_{1},D_{2}\}$ can be assumed to be a nonorientable-admissible datum
such that $d>1$ is odd and non-prime, moreover $\nu(D_{i})<d-1$ and, hence $d-1<\nu(\mathscr{D})<2(d-1)$. Since $\nu(\mathscr{D})$ is even and $d$ is odd then $\nu(\mathscr{D})>d.$ 
Moreover, since $d$ is a non-prime odd integer then $d \geq 9$. But the case $d=9$ is completely solved in \cite{BN&GDL2}. Because of this, it suffices to study the case $d>9$.
Nevertheless we give a complete proof for all $d\geq 3$ based on Lemma \ref {fundlem2}, which is interesting in its own right.
  
\begin{notation}
With  notation and conditions   of Lemma \ref{algEKS}, we will write $$\overline{\beta}:=EKS(\overline{d},\overline{\lambda},\overline{D_{2}})$$ to mean:
$\overline{\beta}$ is the permutation obtained by applying Lemma \ref{algEKS}. The notation ``EKS'' comes from the first letters of the names of
 the authors of \cite{EKS}.
\end{notation}

\begin{Rm}\label{pi}
Let $\Omega$ be a set with $d$ elements and suppose $\overline{\Omega} \subset \Omega$ be a proper subset with $\bar{d}$ elements, i.e. $\bar{d}<d$. Notice that there exists a projection (not a homomorphism):
\begin{eqnarray}
\begin{matrix}
\wp: &\Sym_{\Omega}& \rightarrow &\Sym_{\overline{\Omega}}\\
& \lambda & \longmapsto & \wp(\lambda)
\end{matrix}
\end{eqnarray}
where $\wp(\lambda)$ is obtained from $\lambda$ by deleting the elements of $\Omega \setminus \overline{\Omega}$ in the cyclic decomposition of $\lambda$ (see beginning of subsection \ref{subsec:permut}).
There exists also a monomorphism:
\begin{eqnarray} \label {eq:mono}
\begin{matrix}
\imath: &\Sym_{\overline{\Omega}}& \rightarrow &\Sym_{\Omega}\\
& \bar{\beta} & \longmapsto & \imath(\bar\beta)
\end{matrix}
\end{eqnarray}
such that
\begin{eqnarray}
[\imath( \bar{\beta})](i):= \left\{ \begin{array}{ll}
i, & \textrm{ if $i \notin \overline{\Omega}$},\\
 \bar{\beta}(i), & \textrm { if $i \in \overline{\Omega}.$}
 \end{array}\right.
 \end{eqnarray}
 Consider $\lambda \in \Sym_{\Omega}$ and $\bar{\beta} \in \Sym_{\overline{\Omega}}$. We want to understand how to obtain $\lambda \imath(\bar{\beta})$ from $\wp(\lambda)\bar{\beta}$ and $\lambda$.
 Suppose  $\overline{\Omega}:=\{{\bf w}_1,\dots,{\bf w}_{\bar{d}}\}$ and
 $$\lambda=({\bf w}_1\; \underbrace{u_{1,1} \dots u_{1,t_{1}}}_{S_1}\; {\bf w}_2 \; \underbrace{u_{2,1}\dots u_{2,t_2}}_{S_2} \dots {\bf w}_{\bar{d}}\;\underbrace{u_{\bar{d},1} \dots u_{\bar{d},t_{\bar{d}}}}_{S_{\bar{d}}})$$
 where for $i=1,\dots,\bar{d}$ we have $u_{i,j}\in\Omega\setminus\overline\Omega$ for $j=1,\dots,t_i$ and
 $t_i\geq 0$, i.e. the sequences $S_i=u_{i,1}\dots u_{i,t_i}$ can be empty and $\bar{d}+\sum_{i=1}^{\bar{d}}t_i\leq d.$
 Then $$\wp(\lambda)=({\bf w}_1\; {\bf w}_2 \dots {\bf w}_{\bar{d}}).$$
 Suppose also that $$\wp(\lambda)\bar{\beta}=({\bf w}_{\phi(1)}\dots  {\bf w}_{\phi(v_1)})({\bf w}_{\phi(v_1+1)}\dots {\bf w}_{\phi(v_1+v_2)})\dots ({\bf w}_{\phi(v+1)}\dots {\bf w}_{\phi(v+v_x)})$$
 where $ v:=v_1+v_2+\dots+v_{x-1}$, $ v+v_x=\bar{d}$ and $\phi \in \Sym_{\bar{d}}.$ Then
 \begin{eqnarray}\label{Omega}
 [\lambda \imath(\bar{\beta})](q)= [\imath(\bar{\beta})](\lambda(q))=\left\{ \begin{array}{ll}
\bar{\beta}(\lambda(q)), & \textrm{ if $\lambda(q) \in \overline{\Omega}$},\\
 \lambda(q), & \textrm { if $\lambda(q) \notin \overline{\Omega},$}
 \end{array}\right.
 \end{eqnarray}
but for $\lambda(q)\in\overline\Omega$ we have
\begin{eqnarray*}
 \bar{\beta}(\lambda(q))=[\wp(\lambda)]^{-1}
[\wp(\lambda)\bar{\beta}](\lambda(q))=
\end{eqnarray*}
\begin{eqnarray} \label {eq:10}
[\wp(\lambda)\bar{\beta}] ([\wp(\lambda)]^{-1}(\lambda(q)))=\left\{ \begin{array}{ll}
{[\wp(\lambda)\bar{\beta}](q)}, & \textrm{if $q \in \overline{\Omega},$}\\
{[\wp(\lambda)\bar{\beta}]({\bf w}_{l})}, & \textrm{if $q \notin \overline{\Omega}$ and $q \in S_l.$}
 \end{array}\right.
 \end{eqnarray}
Hence
$$\lambda \imath(\bar{\beta})=({\bf w}_{\phi(1)}\; S_{\phi(1)}\dots {\bf w}_{\phi(v_1)}\; S_{\phi(v_1)})
({\bf w}_{\phi(v_1+1)}\; S_{\phi(v_1+1)}\dots{\bf w}_{\phi(v_1+v_2)}\;S_{\phi(v_1+v_2)})\dots$$$$\dots ({\bf w}_{\phi(v+1)}\;S_{\phi(v+1)}\dots {\bf w}_{\phi(v+v_x)}\;S_{\phi(v+v_x)}).$$
Thus we obtain the cyclic decomposition of $\lambda \imath(\bar{\beta})$ from the cyclic decomposition of $\wp(\lambda)\bar{\beta}$ by replacing each ${\bf w}_{i}$ (determined via $\lambda$) by ${\bf w}_i S_i,$ for $i=1,\dots,\bar{d}.$

If $\lambda$ is a product of several disjoint cycles, we use the procedure from above on each cycle to obtain $\lambda \imath(\bar \beta)$. Namely,    we first
define the sequence $S_i$ which is the consecutive of $\bf{w}_i$, using the
procedure above to the cycle
  which contains ${\bf w}_i$. Then define     $\lambda \imath(\bar \beta)$ as above.  Moreover, if $\lambda$ contains cycles whose elements are totally contained in $\overline{\Omega}$, by (\ref{Omega}) and (\ref {eq:10}) the corresponding cyclies of $\wp(\lambda)\bar{\beta}$ are cycles of $\lambda \imath(\bar{\beta})$.
\end{Rm}

\begin{Lemma} \label {fundlem2}
Let $\mathscr{D}=\{D_1,D_2\}$ be a pair of partitions of an odd integer
$d\ge3$ such that $d-1<\nu(\mathscr{D})\equiv0\pmod2$.
Given a permutation $\lambda\in D_1$, there exists $\beta\in D_2$ such that $\lambda\beta$ is a $(d-2)$--cycle and the
permutation group $G:=\langle\lambda,\beta\rangle\leq\Sigma_d$ is transitive. For any such a collection of permutations, the permutation group $G$ is primitive.
\end{Lemma}

 \begin{proof}
In the case of $d=3$, we have $D_1=D_2=[3]$, and $\beta=\lambda^{-1}$ has the desired
properties. From now on, we will assume that $d>3$.

Let $\nu(\mathscr{D})=(d-1)+r$, with $r>0$ even. Due to Proposition \ref {d impar 1}, we can suppose
\begin{eqnarray}
D_{1}=[c_{1},c_{2},\dots,c_{t}], \quad D_{2}=[d_{1},d_{2},\dots,d_{s},\underbrace{1,\dots,1}_{\ell}],\textrm{ with}\\
1<t<d,\quad 1<s+\ell<d,\quad d-1>\nu(D_{1})\geq\nu(D_{2}), \textrm{ and}\\
c_{1}\geq c_{2}\geq \dots \geq c_{t} ,\quad  d_{1}\geq d_{2}\geq \dots \geq d_{s}>1.
\end{eqnarray}
Let $\Omega=\{1,\dots,d\}$
 and $\lambda \in \Sym_{\Omega}$ be a permutation with the cyclic decomposition
 $\lambda=\lambda_{1}\dots\lambda_{t} \in D_{1}$,  where $\lambda_{i}:=(a_{i,1}\;a_{i,2}\;\dots\;a_{i,c_{i}})$
is a $c_{i}$-cycle, for $i=1,\dots,t$.

Since $\nu(D_{1})+\nu(D_{2})\ge d+1$ with $\nu(D_{1})\geq\nu(D_{2})$, we have
\begin{eqnarray} \label {eq:9}
\nu(D_{1})\ge(d+1)/2, \textrm{ and } c_{1}\geq 3.
\end{eqnarray}

Here is the plan of the proof.
In order to construct $\beta \in D_{2}$, we will divide the problem into three cases. In
each case we want to define the first cycle $\beta_1$ in a cyclic decomposition
$\beta=\beta_1\dots\beta_s$ of $\beta$, in order to guarantee that
$|\Fix(\lambda\beta)|\ge2$ (i.e.\ there exist at least two fixed elements
of the permutation $\lambda\beta$). Then we want to study a related problem in the
symmetric group of order $d-2$. Namely, we need a stronger version of the Lemma \ref{algEKS}.
 
If we are able to make the construction in such a way that
the first cycles $\overline{\lambda}_1$ and $\overline{\beta}_1$ in a solution of this related problem have
supports with non-empty intersection, then we are able to solve the original problem in
the symmetric group of order $d$.\\

{\bf First case:} $ c_{1}+d_{1} > 6$ and $d_{1}\geq 3$.

{\it Step 1.} Define $$\beta_{0}:=(a_{1,3}\;a_{1,2}\;a_{1,1}), \textrm{ then } \lambda \beta_{0}=(\overline{a_{1,1}})(\overline{a_{1,2}})(a_{1,3}\; \dots \; a_{1,c_{1}})\lambda_{2}\dots \lambda_{t},$$
where $\overline{a_{i,j}}$,  means ``used elements" in the sense that they will be
elements in the support of the $d_{1}$-cycle $\beta_1$ of $\beta$ and we cannot use them
to define other cycles of $\beta$. Let $\overline{d}:=d-2$ and let
$$
\overline{D}_{1}:=[c_{1}-2,c_{2},\dots,c_{t}],\quad\overline{D}_{2}:=[d_{1}-2,d_{2},\dots,d_{s},\underbrace{1,\dots,1}_{\ell}]
$$
be partitions of $\overline{d}$. Then
     \begin{eqnarray}\label{b}
 \nu(\overline{D}_{1})+\nu(\overline{D}_{2})=\nu(\mathscr{D})-4=(\overline{d}-1)+(r-2)\equiv \overline{d}+1 \pmod{2}.
     \end{eqnarray}
Put $\overline{\Omega}:=\Omega\setminus\{a_{1,1},a_{1,2}\}$.
Then 
\begin{equation} \label {eq:cyc:dec}
\wp(\lambda\beta_0)=\wp(\lambda)=\wp(\lambda_1)\lambda_2\cdots \lambda_t
\end{equation}
is a cyclic decomposition of $\wp(\lambda)$.

{\it Step 2.} On this step, we will work in the permutation group $\Sym_{\overline \Omega}$. We will prove that the permutation $\bar\beta\in\overline D_2$ from Lemma \ref {algEKS} can be chosen in such a way that $a_{1,3}$ belongs to the support of a $(d_1-2)$-cycle of a cyclic decomposition of $\bar\beta$. Consider two subcases.

{\it Subcase 1:}
$d_1\ge3$ and $r=2.$
 In this subcase, $\nu(\overline{D}_{1})+\nu(\overline{D}_{2})=\overline{d}-1$  and $t=\nu(\overline{D}_{2})+1,$ hence
$$
t=(d_{1}-2)+(d_{2}-1)+\dots+(d_{s}-1)
$$
and this suggests how to define $\bar{\beta}\in \overline {D}_{2}$ such that $\wp(\lambda)\bar{\beta}$ will be a $\bar{d}$-cycle. First we define the $(d_1-2)$-cycle $\bar
\beta_1$ of $\bar{\beta}$ by using for its support the element $a_{1,3}$ and one element
in $\Supp(\lambda_i)$ for $i=2,\dots,d_1-2$, if $d_1-1>3$ or  $\overline{\beta}_1=(a_{1,3})$ if $d_1=3.$
So in the product $\wp(\lambda)\bar{\beta_1}$, the elements of the supports of the first $(d_1-2)$ cycles in the cyclic decomposition (\ref {eq:cyc:dec}) form the support of a
($(\sum_{j=1}^{d_{1}-2}c_{j})-2$)-cycle $\Lambda_{1}$. In other words, $\wp(\lambda)\bar{\beta_1}$ will have
the cyclic decomposition of
the form
 $$
\wp(\lambda)\bar\beta_1 = \Lambda_1 \lambda_{d_1-1}\dots\lambda_t,
 $$
with the number of cycles being $d_{2}+(d_3-1)\dots+(d_{s}-1)$, 
and $\Supp(\Lambda_1)=\Supp(\wp(\lambda_1))\cup\cup_{k=2}^{d_1-2}\Supp(\lambda_k)$.

Suppose that $1\le i\le s-1$. Let $k_{i}:=(d_{1}-2)+\sum_{j=2}^{i}(d_{j}-1)$ and
$m_{i}:=(c_1-2)+\sum_{l=2}^{k_{i}} c_{l}$. Suppose that we have constructed a
$(d_1-2)$-cycle $\bar\beta_1$ and $d_j$-cycles $\bar\beta_j$ with $2\le j\le i$ having
the following properties: a cyclic decomposition (see subsection \ref{subsec:permut}) of
the permutation $\wp(\lambda)\bar\beta_1\dots\bar\beta_j$ has the form
\begin{equation}\label {eq:*}
\wp(\lambda)\bar\beta_1\dots\bar\beta_j = \Lambda_j
\lambda_{k_j+1}\dots\lambda_{k_{j+1}}\dots\lambda_t, \quad 1\le j \le i,
\end{equation}
and
\begin{equation}\label {eq:**}
\Supp (\bar\beta_j) \cap \Supp (\bar\beta_k) = \varnothing , \quad 1\le j < k \le i,
\end{equation}
where $\Lambda_j$ is a $m_j$-cycle with
$\Supp(\Lambda_j)=\Supp(\wp(\lambda_1))\cup\cup_{k=2}^{k_j}\Supp(\lambda_k)$. We want to
determine a $d_{i+1}$-cycle $\bar\beta_{i+1}$ with the similar properties. By using the
procedure of the constructing $\bar\beta_1$, we construct $\bar\beta_{i+1}$ such that its
support contains one element from the support of each of the cycles $\Lambda_i,
\lambda_{k_i+1},\dots\lambda_{k_{i+1}}$, and does not intersect the supports of any of
$\bar\beta_1,\dots,\bar\beta_i$. In order to be able to make this construction for each
$i=1,\dots,s-1$, it is sufficient to show the following inequality:

\begin{eqnarray}\label{sobra}
\vert \Supp(\Lambda_{i}) \setminus \Supp(\bar{\beta}_{1}\dots\bar{\beta}_{i})\vert \ge1,
\end{eqnarray}
where
$$
 \vert \Supp(\Lambda_{i}) \setminus \Supp(\bar{\beta}_{1}\dots\bar{\beta}_{i})\vert
 =m_{i}-(d_1-2+d_2+\dots+d_{i})
 =m_{i}-(k_{i}+i-1)
$$
\begin{eqnarray*}
 =c_1-2+(\sum_{l=2}^{k_{i}}c_{l})-i+1
 =((c_1-3)+\sum_{l=2}^{k_{i}}(c_{l}-1))-i+1. 
\end{eqnarray*}

To prove the inequality 
(\ref{sobra}),
let $n$ be the number of all members of the partition $\overline{D}_1$, which are greater than $1$. We have two cases:
 \begin{enumerate}
  \item if $n\ge k_i$ then $((c_1-3)+\sum_{l=2}^{k_{i}}(c_{l}-1))-i+1=$  $$(c_{1}-3)+\sum_{l=2}^{i}(c_{l}-2)+\sum_{l=i+1}^{k_i}(c_{l}-1) \geq$$
$$\geq 0+0+1=1,$$

\item if $n<k_i$ then $\nu(\overline D_1)=c_1-3+\sum_{l=2}^{n}(c_{l}-1)=c_1-3+\sum_{l=2}^{k_i}(c_{l}-1)$, thus $$((c_1-3)+\sum_{l=2}^{k_{i}}(c_{l}-1))-i+1=$$
$$\nu(\overline{D}_{1})-i+1=$$
$$(s+\ell)-i\geq (s+\ell)-(s-1)=\ell+1 \geq 1,$$
\end{enumerate}
so the inequality (\ref {sobra}) holds.

Define $\bar{\beta}:=\bar{\beta_1}\dots\bar{\beta_s}$, then $k_s=t$ and
$|\Supp(\Lambda_s)|=m_s=(c_1-2)+\sum_{l=2}^{t}c_l=d-2=\bar d$, thus
 \begin{equation} \label {eq:***}
 \wp(\lambda)\bar{\beta}=\Lambda_s
 \end{equation}
is a $\bar{d}$-cycle by (\ref {eq:*}) for $j=i=s$, and a cyclic decomposition of
$\bar\beta$ is $\bar\beta=\bar\beta_1\dots\bar\beta_s$ by (\ref {eq:**}) for $i=s$.

\begin{Rm}
If $S\subset \Fix(\bar{\beta})$, $$\Lambda_s=(u_1 v_1 \dots w_1\;u_2 v_2 \dots w_2 \ \dots \ \dots \ \dots \ u_z v_z \dots w_z\; \ast v_\ast \dots w_{\ast}),$$ where $\{u_1,\dots,u_z\}=S.$ This is going to be useful in {\it Subsubcase 2b}.
\end{Rm}

{\it Subcase 2:} $d_1\ge3$ and $r>2$. 

We start by reordering, in an increasing way, the entries of  $$\overline{D}_{2}=[d_1-2,d_2,\dots,d_s,1,\dots,1],$$ i.e. put
 $$\overline{D}_{2} =[e_1,\dots,e_{\ell},e_{\ell+1},\dots,e_{\ell+s}]$$
 with $1=e_1=\dots=e_{\ell}< e_{\ell + 1} \leq e_{\ell+2} \leq \dots \leq e_{\ell+s}.$  Since $r>2$ and
 $ \nu(\overline{D}_{1})+\nu(\overline{D}_{2})=(\overline{d}-1)+(r-2)$, there is $0\leq k < s$ such that
 $$\nu(\wp(\lambda))+(e_{\ell+1}-1)+\dots +(e_{\ell+k}-1) \leq \bar{d}-1$$ and $$\nu(\wp(\lambda))+(e_{\ell+1}-1)+\dots+(e_{\ell+k}-1)+(e_{\ell+k+1}-1) > \bar{d}-1.$$
 Let $f > 0$ such that
 \begin{eqnarray}\label{a}
\nu(\wp(\lambda))+(e_{\ell+1}-1)+\dots +(e_{\ell+k}-1)+(f-1) = \bar{d}-1.
 \end{eqnarray} Define
 $$\overline{D}_{2,1}:=[e_1,\dots,e_{\ell},e_{\ell+1},\dots,e_{\ell+k},f,\underbrace{1, \dots,1}_{z}],$$ a partition of $\bar{d}$ where $z:=\bar{d}-(\sum_{i=1}^{\ell+k}e_i)-f$ is bigger than zero, and define
$$\overline{D}_{2,2}(j):=[e_{\ell+k+1}, \dots,e_{\ell+j-1}, e_{\ell+j}-f+1,e_{\ell+j+1},\dots,e_{\ell+s}],$$
 a partition of $z+1$, where $j\in\{k+1,\dots,s\}$. Later, we will make a choice of $j$.

Here is the plan of constructing $\bar\beta$ in Subcase 2. 

Firstly, we will define $\beta' \in \overline{D}_{2,1}$ and $\beta'' \in \overline{D}_{2,2}(j)$ such that $\Supp(\beta')\cap\Supp(\beta'')$ consists of only one element denoted by $*\in\overline\Omega$. Secondly, we will put $\bar{\beta}:=\beta'\imath_2(\beta'')$, where $\imath_2$ is a monomorphism from $\Sym_{\underline{\Omega}}$ into $\Sym_{\overline{\Omega}}$, where $\underline{\Omega}$ will be a set with $z+1$ element. Then we will check that $\bar\beta$ has the desired properties.

{\it Defining $\beta'\in \overline{D}_{2,1}$.} By (\ref{a}) we have $\nu(\overline D_1)+\nu(\overline D_{2,1})=\bar{d}-1$. Thus
 \begin{eqnarray}\label{t}
 t=\nu(\overline{D}_{2,1})+1.
 \end{eqnarray}

 By (\ref{t}) we have
 \begin{eqnarray}\label{tt}
 \begin{matrix}
  t=(e_{\ell+1})+(e_{\ell+2}-1)+\dots +(e_{\ell+k}-1)+(f-1),
  \end{matrix}
  \end{eqnarray}
thus the permutation $\wp(\lambda)$ and the partition $\overline{D}_{2,1}$ satisfy the hypothesis of Subcase 1. Consider two subsubcases of Subcase 2.

{\it Subsubcase 2a:} the integer $d_1-2$ is a member of the partition $\overline{D}_{2,1}$, more precisely
$d_1-2=e_{j_0}$ for some $j_0\in\{1,\dots,\ell+k\}$.
We define $\beta'\in\overline{D}_{2,1}$ to be the permutation $\bar{\beta}$ constructed in Subcase 1 via the permutation $\wp(\lambda)$, the partition $\overline{D}_{2,1}$ and its distinguished entry $d_1-2$. 

{\it Subsubcase 2b:} the integer $d_1-2$ is not a member of the partition $\overline{D}_{2,1}$, i.e.\
$d_1-2=e_{j_0}$ for some $j_0>\ell+k$.

We define $\beta'\in\overline{D}_{2,1}$ to be the permutation constructed in Subcase 1 via the permutation $\wp(\lambda)$, the partition $\overline{D}_{2,1}$ and its distinguished entry $f$,
i.e.\ $\beta'\in\overline D_{2,1}$ satisfies the assertions of Lemma \ref {algEKS} and  if $f>1$, $a_{1,3}$ belongs to the support of an $f$-cycle of a cyclic decomposition of $\beta'$, otherwise $a_{1,3}\in \Fix(\beta')$.

In any of the subsubcases {\it 2a} and {\it 2b}, a cyclic decomposition of $\beta'$ has the form $\beta'= \epsilon_1 \dots \epsilon_{\ell+k} \epsilon_f$ where $\epsilon_i$ is a $e_i-$cycle
for $i=1,\dots,\ell+k$ and $\epsilon_f$ is an $f-$cycle with $a_{1,3}\in\Supp(\epsilon_f)$ if $f >1$, otherwise $a_{1,3}\in \Fix(\beta')$.
Let 
$$
F := \{\epsilon_1,\dots,\epsilon_{\ell}\} \subset {\overline{\Omega}}\setminus \{a_{1,3}\} 
$$ 
be the set of $\ell$ elements corresponding to the 1-cycles $\epsilon_1,\dots,\epsilon_{\ell}.$

Remark that, in Subsubcase {\it 2a}, we have $a_{1,3}\in\Supp(\epsilon_{j_0})$, and we can choose any $j$ in order to define $\beta''$.
In Subsubcase {\it 2b}, we have $a_{1,3}\in\Supp(\epsilon_f)$ if $f >1$, otherwise $a_{1,3}\in \Fix(\beta')$,  and define $j:=j_0-\ell$ and $\ast:=a_{1,3}$.

{\it Defining $\beta''\in \overline{D}_{2,2}(j)$.} 
Here, we will define $\beta''\in \overline{D}_{2,2}(j)$ 
as a permutation of the set 
$$
\underline{\Omega}:=\{\ast, u_1,\dots,u_z\}
$$
of $z+1$ elements formed by the disjoint union of the element $\ast$  and the set 
$$
\{u_1,\dots,u_z\}:=\overline{\Omega} \setminus (F\cup\Supp( \epsilon_1 \dots \epsilon_{\ell+k} \epsilon_f)\cup \{a_{1,3}\}).
$$
Notice that
 $\nu(\overline{D}_{2})=\nu(\overline{D}_{2,1})+\nu(\overline{D}_{2,2}(j))$,  so  by (\ref{b}) and
 (\ref{a}) we have
  $\nu(\overline{D}_{2,2}(j))=r-2$.

In order to define $\beta''$, 
consider $$\tau \in \overline{D}_{2,2}(j)=[e_{\ell+k+1}, \dots,e_{\ell+j-1}, e_{\ell+j}-f+1,e_{\ell+j+1},\dots,e_{\ell+s}]$$  in  $\Sym_{\underline{\Omega}}$,
a permutation such that $\ast$ is in its $(e_{\ell+j}-f+1)$-cycle. Since $r>2$ is even, $\beta''$ will be a non-trivial even permutation and we can write $\beta''$  as a product of two $(z+1)$-cycles by Lemma 3.1 of  \cite{EKS}. Put $\tau=\sigma \gamma$, where $\sigma$ and $\gamma$ are $(z+1)-$cycles.

As observed in the end of Step 1, 
 the $\bar{d}$-cycle $\wp(\lambda)\beta'$ has the form
\begin{eqnarray}
(u_1 v_1 \dots w_1\;u_2 v_2 \dots w_2 \ \dots \ \dots \ \dots \ u_z v_z \dots w_z\; \ast v_\ast \dots w_{\ast})
\end{eqnarray}
where $\{v_1, \dots, w_1,v_2, \dots, w_2,\dots \ \dots,v_z, \dots, w_z,v_\ast, \dots, w_\ast\}=\overline\Omega\setminus\underline{\Omega}$. By Lemma 3.3 of \cite{Ez}
there exists $\eta \in \Sym_{\underline{\Omega}}$ such that $\eta \sigma^{-1} \eta^{-1}=(u_1\; u_2\dots u_z\; \ast)$ and $\eta(\ast)=\ast$. Define 
$$
\beta'':=\eta \tau \eta^{-1} \in \overline{D}_{2,2}.
$$

{\it Defining $\bar\beta\in \overline{D}_{2}$.} 
Put $\bar{\beta}:=\beta'\imath_2(\beta''),$ where $\imath_2:\Sym_{\underline{\Omega}}\to\Sym_{\overline{\Omega}}$ is the monomorpism (\ref {eq:mono}). 

Let us check that $\bar\beta\in\overline{D}_2$. Notice that $\beta'$ fix  $\underline{\Omega}\setminus \{\ast\}$,  then the product $\beta'\imath_2(\beta'')$
contains the cycles of $\beta''$ except its  $e_{\ell+j}-f+1$ cycle.    Similarly  $\imath_2(\beta'')$ fix the elements which corresponds to the cycles $\epsilon_i$
for $i=1,\dots,\ell+k$, then the product $\beta'\imath_2(\beta'')$  contains the cycles of $\beta'$ except its  $f-$ cycle. Finally the product restricted to the union
of the elements in the $f-$ cycle of $\beta'$ and the elements in the  $(e_{\ell+j}-f+1)-$ cycle  of $\beta''$ gives the cycle of length $e_{\ell+j}$ of $\bar{\beta}$.

Let us check that $a_{1,3}$ belongs to the support of a $(d_1-2)$-cycle of a cyclic decomposition of $\bar\beta$. 
In subsubcase {\it 2a}, we have $a_{1,3}\in\Supp(\epsilon_{j_0})$, and we can choose any $j$ in order to define $\beta''$.
In subsubcase {\it 2b}, we have $a_{1,3}\in\Supp(\epsilon_f)$ if $f >1$, otherwise $a_{1,3}\in \Fix(\beta')$,  and we define $j:=j_0-\ell$ and $\ast:=a_{1,3}$.

Let us check that $\wp(\lambda)\bar{\beta}$ is a $\bar{d}$-cycle.
The product $ \wp(\lambda)\bar{\beta}=(\wp(\lambda)\beta')\imath_2(\beta'')$ can be described as being obtained from  $(u_1 \dots u_{z}\;\ast)\beta''_2$ by replacing
each $u_j$ and $\ast$ by the sequences $u_j v_j\dots w_j$ and $\ast v_\ast\dots w_\ast$ respectively, occurring in $\wp(\lambda)\bar\beta_1$, as explained in Remark \ref{pi}.
Thus $\wp(\lambda)\bar{\beta}$ is a $\bar{d}$-cycle.

{\it Step 3.} On this step, we show how to construct $\beta\in\Sym_\Omega$ via $\bar\beta\in\Sym_{\overline\Omega}$ constructed on Step 2.
By Step 2, 
a cyclic decomposition of $\bar\beta$ has the form $\bar\beta=\bar\beta_1\dots\bar\beta_s$, for a $(d_1-2)-$cycle $\bar\beta_1$ and $d_j$-cycles $\bar\beta_j$ with $2\le j\le s$, where $a_{1,3}\in\Supp(\bar{\beta}_1)$ or $(a_{1,3})$ is the $1-$cycle which corresponds to $\bar{\beta}_1$.

Put $\beta:=\beta_{0}\imath(\bar{\beta})$. Observe that $\beta_0$ is a 3-cycle and
$\bar\beta_1$ is a $(d_1-2)$-cycle, moreover if $d_1>3$ then $\{a_{1,3}\}=\Supp(\beta_{0}) \cap
\Supp(\imath(\bar\beta_1))$. 
 In any case, 
it follows that $\beta_1:=\beta_{0}\imath(\bar\beta_1)$
is a $d_1$-cycle with
$$
a_{1,3} \in \Supp(\beta_1).
$$
 
Therefore the permutation $\beta =\beta_{0}\imath(\bar{\beta})=\beta_1 \imath(\bar\beta_2
\dots \bar\beta_s) \in D_{2}$, since the cycles $\bar\beta_1,\dots,\bar\beta_s$ are disjoint 
and $\Supp(\beta_{0}) \cap \Supp(\imath(\bar\beta_i))=\varnothing$ for $2\le i\le s$. Observe that $\lambda
\beta_0=\imath(\wp(\lambda))$, which follows via the inclusion
$\Supp(\beta_0)\subseteq\Supp(\lambda_1)$. Therefore, in $\Sym_{\Omega},$ we have 
$$
\lambda\beta=(\lambda \beta_0)\imath(\bar{\beta})=\imath(\wp(\lambda)\bar{\beta})
$$ 
is a $(d-2)$-cycle by Step 2,
as required.\\

     {\bf Second case: $c_1=d_1=3$}

     \begin{itemize}
 \item If $t=2$, since $c_1=3$ and it is the biggest summand of $D_1,$ then $d \leq 6$.

   If $d=5$, from the table in the Appendix, line 1, the result follows.

\item If $t=3$, we have $5\leq d \leq9$.
 If $d=5$, then $D_1=D_2=[3,1,1]$. But in this case $\nu(\mathscr{D})=d-1$, a contradiction with one hypothesis on the total defect. For
 $d=7$ and $d=9$, the possibilities for $D_1, D_2$, as well as the realization for these cases, are given in the table in the Appendix, lines 2, 3, 4 and 5, 6, 7, respectively.
\item If $t\geq 4$   then   we have either
$D_{1}=D_{2}=[3,3,2,1]$, or $D_{1}=[3,3,2,1]$ and $D_{2}=[3,2,2,2]$, or $d=11=\nu(\mathscr{D})-1$, or $c_{i} \geq 2$ for $i=1,2,3,4$.

 For the first two cases,  the result follows from the table in the Appendix, lines 8, 9 and 10.

  Let $d=11=\nu(\mathscr{D})-1$ .  Since $\nu(D_1)\geq (d+1)/2$ then $\nu(D_1)\geq 6$ and $t\leq 5$. Then $t$ is either $4$ or $5$.
The possibilities for $D_1,\; D_2$,  as well the realization for these cases, are given by the table in the Appendix, lines 11 to 19.\\

 So from now on let us assume that $t\geq 4$  and   $c_{i} \geq 2$ for $i=1,2,3,4$.

Define
   $$\beta_{0}:= (a_{1,2}\; a_{1,1}\; a_{*})(a_{2,2}\;a_{2,1}\; a_{3,1})$$ where $a_{*}$ exists just if $d_{2}=3$ and in this case
   $a_{*}:=
    a_{4,1}.$
  Then $\lambda\beta_{0}=$
 
 \begin{small}
  $$=\left\{ \begin{array}{ll}
  (\overline{a_{1,1}})(\overline{a_{1,2}}\; a_{1,3})(\overline{a_{2,1}})(\overline{a_{2,2}} \dots a_{2,c_{2}} \; \overline{a_{3,1}} \dots a_{3,c_{3}})\;\lambda_4  \dots \lambda_{t}, \textrm{ if  $d_{2}=2$,}\\
  (\overline{a_{1,1}})(\overline{a_{1,2}}\; a_{1,3}\;\overline{a_{4,1}} \dots a_{4,c_{4}})(\overline{a_{2,1}})(\overline{a_{2,2}} \dots a_{2,c_{2}} \; \overline{a_{3,1}} \dots a_{3,c_{3}})\;\lambda_5 \dots \lambda_{t}, \textrm{ if  $d_{2}=3$,}\\
  \end{array}\right. $$
 \end{small}
 Notice that $d_{2}>1$, if not $\nu(D_{2})=2$ and $\nu(D_{1})= d-1$, a contradiction with the hypothesis over $\nu(\mathscr{D})$ in the present section.

  Let $\overline{d}:=d-{(d_{1}+d_{2})}$ and let
  $$\overline{D_{1}}:=\begin{cases}
  [(c_{1}-2),(c_{2}-2)+(c_{3}-1),c_4,\dots,c_{t}], \textrm{ if $d_{2}=2$,}\\
[(c_{1}-2)+(c_{4}-1),(c_{2}-2)+(c_{3}-1),c_5,\dots,c_{t}], \textrm{ if  $d_{2}=3$,}\\
  \end{cases} $$ and
  $$\quad  \overline{D_{2}}:=[d_{3},\dots,d_{s},1,\dots,1]$$ be partitions of $\overline{d}$.
  Then $\nu(\overline{D}_{1})+\nu(\overline{D}_{2})=\nu(\mathscr{D})-(d_{1}+d_{2}+2)=(\overline{d}-1)+(r-2)\equiv \overline{d}+1 \pmod{2}$.
  Put $\overline{\Omega}:=\Omega-\{a_{1,1},a_{1,2},a_{2,1},a_{2,2},a_{3,1},a_{*}\}$.

  Consider   in $\Sym_{\overline{\Omega}}:$

  $$\wp(\lambda\beta_{0})=\left\{ \begin{array}{ll}
  ( a_{1,3})(a_{2,3}\dots a_{2,c_{2}} \; a_{3,2} \dots a_{3,c_{3}})\; \lambda_4 \dots \lambda_{t}, \textrm{ if  $d_{2}=2$,}\\
  ( a_{1,3}\; a_{4,2} \dots a_{4,c_{4}})(a_{2,3} \dots a_{2,c_{2}} \; a_{3,2} \dots a_{3,c_{3}})\; \lambda_5 \dots \lambda_{t}, \textrm{ if   $d_{2}=3$}  \end{array}\right. $$
 in $\overline{D_{1}}$   and
 $\overline{\beta}:=EKS(\overline{d},\wp(\lambda \beta_0),\overline{D}_{2})$.
By the Lemma \ref{algEKS},
  $\wp(\lambda) \bar{\beta}$ is a $\overline{d}$-cycle.

 Put $\beta:=\beta_{0}\imath(\bar{\beta})$, thus $\lambda \beta$  will be the $(d-2)$-cycle constructed from $\wp(\lambda \beta_0) \bar{\beta}$
  as explained in Remark \ref{pi}.
 \end{itemize}

{\bf Third case:} $ d_{1} =2$. We have
 $s\geq2$. If not  $\nu(D_{2})=1$ and $\nu(D_{2})>d$, impossible.  If $c_{1}$ is either $3$ or $4$, we have $c_{2}\geq 3$, if not $ \nu(D_{1})\leq (d+1)/2 $ and since $\nu(D_{2})\leq (d-1)/2$ then $\nu(\mathscr{D})\leq d$,
 a contradiction with the hypothesis.
  Define $$\beta_{0}:=(a_{1,1}\; a_{1,2})(a_{*}\; a_{\#}),$$ where
 $$
  a_{*}= \left\{ \begin{array}{ll}
a_{2,2}, \textrm{ if $c_{1}\leq 4$}\\
 a_{1,3}, \textrm { if $c_{1}>4$}
 \end{array}\right.
\\ \textrm{ and  }
 a_{\#}= \left\{ \begin{array}{ll}
a_{2,1}, \textrm{ if $c_{1}\leq 4$}\\
 a_{1,4}, \textrm { if $c_{1}>4$}
 \end{array}\right..
$$
Then $$\lambda\beta_{0}=\begin{cases}
(\overline{a_{1,1}})(\overline{a_{1,2}}\;a_{1,3}\dots a_{1,c_{1}})(\overline{a_{2,1}})(\overline{a_{2,2}} \dots a_{2,c_{2}})\lambda_{3}\dots \lambda_{t}, \textrm{ if $c_{1}\leq 4$,}\\
(\overline{a_{1,1}})(\overline{a_{1,3}})(\overline{a_{1,2}}\; \overline{a_{1,4}} \dots a_{1,c_{1}})\lambda_{2}\dots \lambda_{t},  \textrm{ if $c_{1}>4$.}
\end{cases}.$$
  Let $\overline{d}:=d-4$ and let
  $$
   \overline{D}_{1}:=\begin{cases}
  [c_{1}-2,c_{2}-2,c_3,\dots,c_{t}],&\textrm{ if $c_{1} \leq 4$} \\
[c_{1}-4,c_{2},\dots, c_{t}], &\textrm{ if $c_{1} > 4$}
  \end{cases}
 , \quad\overline{D}_{2}:=[d_{3},\dots,d_{s},1,\dots,1]$$
 be partitions of $\overline{d}.$ Then $\nu(\overline{D}_{1})+\nu(\overline{D}_{2})=\nu(\mathscr{D})-6=(\overline{d}-1)+(r-2) \equiv \overline{d}+1 \pmod{2}$. Put $\overline{\Omega}:=\Omega-\{a_{1,1},a_{1,2},a_{*},a_{\#}\}$.   \\

 Whatever is the case, observe that $\overline{\Omega}$ is the set of ``non-used elements'' and define $\overline{\lambda}\in \Sigma_{\overline{\Omega}}$ to be obtained from
$\lambda \beta_{0}$ by removing the ``used elements'' from its cyclic decomposition (thus $\overline{\lambda}\in\overline{D}_{1}$), and   $\overline{\beta}:=EKS(\overline{d},\overline{\lambda},\overline{D}_{2})$.
 Then $\bar{\lambda} \bar{\beta}$ is a $\overline{d}$-cycle. Put $\beta:=\beta_{0}\overline{\beta}$, thus $\lambda \beta$  is the $(d-2)$-cycle constructed from $\bar{\lambda} \bar{\beta}$ by inserting the maximal
 subsequence $\overline{a_{i,j}}\ldots\overline{a_{k,\ell}}$ of ``used elements'' (in any non-trivial cycle of $\lambda \beta_{0}$), next to $a_{m,n}:=(\lambda \beta_{0})^{-1}(\overline{a_{i,j}})$ on the right.

It remains to check that $\beta:=\beta_0\bar{\beta}\in D_2=[d_1,d_2,d_3,\dots,d_s,1,\dots,1]$.  
In the cases 2 and 3, this follows from the fact that the permutations $\beta_0\in[d_1,d_2,1,\dots,1]$ and $\bar{\beta}\in\overline{D}_2=[d_3,\dots,d_s,1,\dots,1]$ have disjoint supports
(consisting of ``used'' and ``non-used elements'' respectively).

 Finally, by considerations in the begining of the proof, the theorem is proved.

  \end{proof}

\paragraph{Proof of Theorem} \ref{maior}

 \begin{proof}
Since $\nu(\mathscr{D})>d-1$, we have $d\ge2$. Since $d$ is odd, we have $d\ge3$.
Let $\mathscr{D}=\{D_1,D_2\}$.
By Lemma \ref {fundlem2}, there exist permutations $\sigma_i\in D_i$ for $i=1,2$ such that $\sigma_1\sigma_2$ is a $(d-2)$--cycle and the permutation group $G:=\langle\sigma_1,\sigma_2\rangle\leq\Sigma_d$ is transitive and primitive.

Since $d-2$ is odd and $\sigma_1\sigma_2$ is a $(d-2)$--cycle, we can define $\alpha:=\sqrt{\sigma_1\sigma_2}$, so $G=\langle\alpha,\sigma_1,\sigma_2\rangle$ and $\sigma_1\sigma_2=\alpha^{2}$.
Now we apply the Hurwitz approach to this triple of permutations.
By Theorem \ref {thm:1.2} and Remark \ref{mon} we can realize $G$ as the monodromy group of a branched covering $(M,\phi, \mathbb{R}P^2, B_{\phi},d)$ whose branch datum is $\mathscr{D}$. Moreover, since
the monodromy group $G$ is transitive and primitive, by Proposition \ref {yo} the covering surface $M$ is connected and the branched covering $\phi$ is indecomposable.
 \end{proof}

\section{Arbitrary number of branch points}

This section is devoted to the proof of the main theorem. We will do that by induction on the number of partitions of  the branch datum. For that, it is necessary to know how we can reduce the size of the branch datum without loose relevant information. This procedure is inspired by the proof of Theorem 5.1\cite{EKS}.

From lemma 4.2\cite{EKS} and lemma 4.3 \cite{EKS} we can read the following two assertions, which we will state not as strong as they were given in \cite{EKS}.

\begin{Lemma}\label{l42}[Lemma 4.2 \cite{EKS}]
Let $A, B$ be partitions of $d$ with $\nu(A) + \nu(B) = d-t, t \ge 1$. Then there exist permutations
$\alpha \in A$ and $\beta \in B$ such that
$\nu(\alpha \beta)=d-t$.
\qed
\end{Lemma}

\begin{Lemma}\label{l43}[Lemma 4.3 \cite{EKS}]
Let $A,B$ be partitions of $d$ with $\nu(A)+ \nu(B) = (d - 1) + r, r > 0$. Then for
each $k$ satisfying $0 \leq k \leq r, k \equiv r \pmod 2$, one may choose $\alpha \in A, \beta \in B$
so that $\nu(\alpha\beta) = (d-1)-k$.
\qed
\end{Lemma}

\begin{Prop} \label {pro:reduc}
Let $\mathscr{D}=\{D_1,\dots,D_s\}$ be a collection of partitions of an odd integer $d\ge3$ such that $s\geq 3$ and
$d-1<\nu(\mathscr{D})\equiv0\pmod2$.
Then there exist a partition $D$ and permutations $\gamma_1,\gamma_2\in\Sigma_d$
such that the new collection of partitions $\widehat{\mathscr{D}}=\{D,D_3, \dots, D_s\}$ satisfies the conditions
$d-1 <\nu(\widehat{\mathscr{D}})\equiv 0 \pmod{2}$ and $\gamma_1\in D_1,\gamma_2\in D_2$, $\gamma_1\gamma_2\in D$.
\end{Prop}

\begin{proof}
Step 1. Without
 loss of generality we may and shall assume that $\sum_{i=3}^s\nu(D_i)>1$. Suppose the contrary. Then $\sum_{i=3}^s\nu(D_i)=1$, $s=3$ and $D_3=[2,1\dots,1]$.
 In this case either $\nu(D_1)>1$ or $\nu(D_2)>1$, since otherwise $D_1=D_2=D_3=[2,1\dots,1]$, $\nu(\mathscr{D})=3$ is odd,

which is impossible. So, we can relabel the partitions in $\mathscr{D}$ in such a way that $\nu(D_3)>1.$

Step 2. Suppose $\nu(\mathscr{D})=(d-1)+2q$ with $q > 0$. Then $\nu(D_1)+\nu(D_2)=(d-1)+2q-\sum_{i=3}^s\nu(D_i)$.

If $r:=2q -\sum_{i=3}^s\nu(D_i)\leq 0$, define $t:=1-r\geq 1$ then $\nu(D_1)+\nu(D_2)=d-t$ and applying Lemma \ref{l42},  there exist $\gamma_1 \in D_1,\; \gamma_2 \in D_2$ such that
$$ \nu(\gamma_1 \gamma_2)=d-t=(d-1)+2q-\sum_{i=3}^s\nu(D_i).$$ Let $D$ be the partition determined by the cycle structure of $\gamma_1 \gamma_2$.
Then $\nu(D)+\sum_{i=3}^s\nu(D_i)= (d-1)+2q $ and $$\nu(\widehat{\mathscr{D}})= \nu(\mathscr{D}),$$ where $\widehat{\mathscr{D}}=\{D,D_3,\dots, D_s\}.$
Moreover, since $d-1 < \nu(\mathscr{D})$ then $d-1 < \nu(\widehat{\mathscr{D}})$.

If $r>0$, let $k \in \{ 0,1\}$ such that $r\equiv k$ mod 2.
Applying Lemma \ref{l43}, there are permutations $\gamma_1 \in D_1,\; \gamma_2 \in D_2$ such that
$\nu(\gamma_1 \gamma_2)=(d-1)-k.$  Let $D$ be the partition determined by the cyclic structure of $\gamma_1 \gamma_2$.
Since  $\nu(D)\equiv\nu(D_1)+\nu(D_2) \pmod{2}$ then $$\nu(\widehat{\mathscr{D}})\equiv \nu(\mathscr{D}) \pmod{2},$$ where $\widehat{\mathscr{D}}=\{D,D_3,\dots, D_s\}.$
    If $k=0$ then $$ d-1<(d-1)+\sum_{i=3}^s\nu(D_i)=\nu(D)+\sum_{i=3}^s\nu(D_i)=\nu(\widehat{\mathscr{D}}),$$
    where inequality is justified because  $\sum_{i=3}^s\nu(D_i)>0$ since $s \geq 3$.  If $k=1$ then
    $
\nu(\widehat{\mathscr{D}})=(d-1)-1 + \sum_{i=3}^s\nu(D_i).
    $
By step 1, $\sum_{i=3}^s\nu(D_i)>1$.
Hence $d-1 < \nu(\widehat{\mathscr{D}}). $
\end{proof}

\begin{FundLem} \label {fundlem}
Let $\mathscr{D}=\{D_1,\dots,D_s\}$ be a collection of partitions of an odd integer $d\ge3$ such that $d-1<\nu(\mathscr{D})\equiv0\pmod2$.
Then there exist permutations $\sigma_i\in D_i$ for $1\le i\le s$ such that $\sigma_1\dots\sigma_s$ is a $(d-2)$--cycle and the permutation group
$G:=\langle\sigma_1,\dots,\sigma_s\rangle\leq\Sigma_d$ is transitive. For any such a collection of permutations, the permutation group $G$ is primitive.
\end{FundLem}

\begin{proof}
Step 1. Let us prove the first assertion of Lemma by induction on the number $s$ of partitions of $d$ in $\mathscr{D}$.
First we observe that $s\ge2$, since $\nu(\mathscr{D})>d-1$.

For two partitions ($s=2$), the first assertion of Lemma is proved
in Theorem \ref{maior} (see the beginning of its proof). For the induction hypothesis, suppose that $k\ge2$ and the assertion of Lemma holds
for $s=k$, i.e.\ for every collection $\mathscr{D}=\{D_1,\dots,D_k\}$ of $k$ partitions of $d$ such that $d-1<\nu(\mathscr{D})\equiv0\pmod2$
there exist permutations $\sigma_i\in D_i$ for $1\le i\le k$ such that $\sigma_1\dots\sigma_k$ is a $(d-2)$--cycle and the permutation group $\langle\sigma_1,\dots,\sigma_k\rangle\leq\Sigma_d$ is transitive.
For the inductive step, consider a collection $\mathscr{D}=\{D_1,D_2,D_3,\dots,D_{k+1}\}$ of $k+1$ partitions of $d$ such that $d-1<\nu(\mathscr{D})\equiv0\pmod2$.
By Proposition \ref {pro:reduc},
there exists a collection $\widehat{\mathscr{D}}=\{D,D_3, \dots, D_{k+1}\}$ of $k$ partitions
such that $d-1 <\nu(\widehat{\mathscr{D}})\equiv 0 \pmod{2}$, moreover there exist permutations $\gamma_1\in D_1,\gamma_2\in D_2$ such that $\gamma_1\gamma_2\in D$.

Thus, by the induction hypothesis,
there exist permutations  $\sigma \in D,\; \sigma_i \in D_i$  for $i=3,\dots,k+1$
in $\Sigma_d$, such that the permutation group $\widehat{G}=\langle
\sigma,\; \sigma_3,\dots,\sigma_{k+1}\rangle$ is transitive and
$\sigma \sigma_3\dots\sigma_{k+1}\in[d-2,1,1]$.
Since $\sigma\in D$ and $\gamma_1\gamma_2\in D$, it is clear that $\sigma$ and $\gamma_1 \gamma_2$ are conjugate. Thus there exists 
$\lambda \in \Sigma_d$ such that $\sigma=\lambda \gamma_1 \gamma_2 \lambda^{-1}$. Define the permutation  group $G=\langle
\lambda \gamma_1 \lambda^{-1},\; \lambda \gamma_2 \lambda^{-1} , \sigma_3,\dots,\sigma_{k+1}\rangle$ where the relation
$\lambda \gamma_1 \gamma_2 \lambda^{-1} \sigma_3\dots\sigma_{k+1}=\sigma \sigma_3\dots\sigma_{k+1}\in[d-2,1,1]$
certainly holds.
Since $\widehat{G} \le G$ and $\widehat{G}$ is transitive, the group $G$ is transitive too. The induction step is completed.

Step 2. Let us prove the second assertion of Lemma.
Since $\sigma_1\dots\sigma_s$ is a $(d-2)$--cycle and the group $G=\langle\sigma_1,\dots,\sigma_s\rangle$ is transitive, it follows from Example \ref {referee} that $G$ is primitive.

\end{proof}

\begin{Rm}
Lemma \ref {fundlem} is equivalent to the following assertion which is more convenient for studying branched coverings of the 2-sphere.
Let $\mathscr{D}=\{D_1,\dots,D_s\}$ be a collection of partitions of an odd integer $d\ge2$ such that
 $2d-2\le\nu(\mathscr{D})\equiv0\pmod2$ and $D_1=[d-2,1,1]$.
Then there exist permutations $\sigma_i\in D_i$ for $1\le i\le s$ such that $\sigma_1\dots\sigma_s$ is the trivial permutation.
\end{Rm}

\begin{theo}\label{maiorgeral}
Let  $\mathscr{D}$ be a collection of partitions of an odd integer $d$
such that $d-1<\nu(\mathscr{D})\equiv0\pmod2$ (compare (\ref {hcpp})).
Then it can be realized as the branch datum of an indecomposable (and hence primitive) branched covering of degree $d$ over the projective plane with a connected covering surface.
\end{theo}

\begin{proof}
Since $\nu(\mathscr{D})>d-1$, we have $d\ge2$. Since $d$ is odd, we have $d\ge3$.
Let $\mathscr{D}=\{D_1,\dots,D_s\}$.
By Fundamental Lemma \ref {fundlem}, there exist permutations $\sigma_i\in D_i$ for $1\le i\le s$ such that $\sigma_1\dots\sigma_s$ is a $(d-2)$--cycle and the permutation group
$G:=\langle\sigma_1,\dots,\sigma_s\rangle\leq\Sigma_d$ is transitive and primitive.

Since $d-2$ is odd and $\sigma_1\dots\sigma_s$ is a $(d-2)$--cycle, we can define $\alpha:=\sqrt{\sigma_1\dots\sigma_s}$, so $G=\langle\alpha,\sigma_1,\dots,\sigma_s\rangle$ and $\sigma_1\dots\sigma_s=\alpha^{2}$.
Now we apply the Hurwitz approach to this collection of permutations.
By Theorem \ref {thm:1.2} and Remark \ref{mon} we can realize $G$ as the monodromy group of a branched covering $(M,\phi, \mathbb{R}P^2, B_{\phi},d)$ whose branch datum is $\mathscr{D}$. Moreover, since
the monodromy group $G$ is transitive and primitive, by Proposition \ref {yo} the covering surface $M$ is connected and the branched covering $\phi$ is indecomposable.
\end{proof}

We observe that Fundamental Lemma \ref {fundlem} is in fact equivalent to the existence, for any collection $\mathscr{D}=\{D_1,\dots,D_s\}$ of partitions of an odd integer $d\ge3$ with $D_1=[d-2,1,1]$, of a branched covering
$(M,\phi, S^2, B_{\phi},d)$ over $S^2$ whose branch datum is $\mathscr{D}$ and the covering surface $M$ is connected, provided that the Hurwitz conditions $2d-2\ge\nu(\mathscr{D})\equiv0\mod2$ hold.
A similar result was obtained by R.~Thom \cite {Thom} about the existence, for any collection $\mathscr{D}=\{D_1,\dots,D_s\}$ of partitions of an integer $d\ge3$ with $D_1=[d]$, of a branched covering
$(S^2,\phi, S^2, B_{\phi},d)$ over $S^2$ whose branch datum is $\mathscr{D}$ and the covering surface $M=S^2$, provided that the Hurwitz condition $\nu(\mathscr{D})=2d-2$ holds.
More precisely, we obtain from Fundamental Lemma \ref {fundlem} the following result that contributes to giving a partial solution of the realization problem for certain branch data \cite{EKS}.

\begin{theo}\label{maior:sphere}
Let $\mathscr{D}=\{D_1,\dots,D_s\}$ be a collection of partitions of an odd integer $d\ge3$ such that $2d-2\le\nu(\mathscr{D})\equiv0\pmod2$ and $D_1=[d-2,1,1]$. Then $\mathscr{D}$ 
can be realized as the branch datum of a branched covering of degree $d$ over the 2-sphere with connected covering surface. Any such a branched covering is indecomposable (and hence primitive).
\end{theo}

\begin{proof}
Consider the collection $\hat{\mathscr{D}}:=\mathscr{D}\setminus\{D_1\}=\{D_2,\dots,D_s\}$ of partitions of $d$. Since $\nu(\mathscr{D})\ge2d-2$, we have $\nu(\hat{\mathscr{D}})=\nu(\mathscr{D})-\nu(D_1)=\nu(\mathscr{D})-d+3\ge d+1$.
 
By applying Fundamental Lemma \ref {fundlem} to $\hat{\mathscr{D}}$, we have that there exist permutations $\sigma_i\in D_i$ for $2\le i\le s$ such that $\sigma_2\dots\sigma_s$ is a $(d-2)$--cycle
and the permutation group $G:=\langle\sigma_2,\dots,\sigma_s\rangle\leq\Sigma_d$ is transitive and primitive.

Take the permutation $\sigma_1:=(\sigma_2\dots\sigma_s)^{-1}$,

so $G=\langle\sigma_1,\dots,\sigma_s\rangle$ and $\sigma_1\dots\sigma_s$ is the trivial permutation.
Now we apply the Hurwitz approach to this collection of permutations.
By Theorem \ref {thm:1.2} and analogue of Remark \ref{mon} for $N=S^2$, we can realize $G$ as the monodromy group of a branched covering $(M,\phi, S^2, B_{\phi},d)$ whose branch datum is $\mathscr{D}$. Moreover, since
the monodromy group $G$ is transitive and primitive, by Proposition \ref {yo} the covering surface $M$ is connected and the branched covering $\phi$ is indecomposable.
\end{proof}

Clearly, Theorem \ref {maior:sphere} implies Fundamental Lemma \ref {fundlem}.

\section {Appendix}
The  table below is related with the second case in the proof of Lemma \ref{fundlem2}, for $d=5,\;7,\;9,\;11.$ We list all the possibilities for $\mathscr{D}=\{D_1,D_2\}$, with $c_1=d_1=3$ and
 $d-1<\nu(\mathscr{D})\equiv 0 \pmod {2}$, as well their realization by permutations $\lambda \in D_1$, $\beta \in D_2$ such that $\alpha \beta$ is a $(d-2)$-cycle.\\

\begin{scriptsize}
\hspace*{-1.5cm}
\begin{tabular}{|c|c|c|c|c|c|c|c|c|c|c|}
\hline
 &$t$& $d$&  $\mathscr{D}=\{D_1,D_2\}$, &   $\lambda \in D_1$ & $\beta \in D_2$ & $\lambda\beta\in [d-2,1,1]$\\
\hline
1&$2$& $5$ &$\{[3,2],[3,2]\}$  &$(1\;2\;3)(4\;5)$ &$(5\;4\;1)(3\;2)$&$(1\;3\;5)(2)(4)$  \\
\hline
2&$3$&$7$&$\{[3,3,1],[3,3,1]\}$&$(1\;2\;3)(4\;5\;6)(7)$&$(3\;2\;4)(6\;5\;7)(1)$&$(1\;4\;7\;6\;3)(2)(5)$\\
\hline
3&$3$&$7$&$\{[3,3,1],[3,2,2]\}$&$(1\;2\;3)(4\;5\;6)(7)$&$(3\;2\;4)(5\;6)(7\;1)$&$(1\;4\;6\;3\;7)(2)(5)$\\
\hline
4&$3$&$7$&$\{[3,2,2],[3,2,2]\}$&$(1\;2\;3)(4\;5)(6\;7)$&$(1\;2)(5\;4\;6)(7\;3)$&$(2\;7\;5\;6\;3)(1)(4)$\\
\hline
5&$3$&$9$&$\{[3,3,3],[3,3,1,1,1]\}$&$(1\;2\;3)(4\;5\;6)(7\;8\;9)$&$(3\;2\;4)(6\;5\;7)(1)(8)(9)$&$(1\;4\;7\;8\;9\;6\;3)(2)(5)$\\
\hline
6&$3$&$9$&$\{[3,3,3],[3,2,2,1,1]\}$&$(1\;2\;3)(4\;5\;6)(7\;8\;9)$&$(3\;2\;4)(6\;5)(7\;1)(8)(9)$&$(1\;4\;6\;3\;7\;8\;9)(2)(5)$\\
\hline
7&$3$&$9$&$\{[3,3,3],[3,3,3]\}$&$(1\;2\;3)(4\;5\;6)(7\;8\;9)$&$(3\;2\;4)(6\;5\;7)(1\;8\;9)$&$(1\;4\;7\;9\;6\;3\;8)(2)(5)$\\
\hline
8&$4$&$9$&$\{[3,3,2,1],[3,3,2,1]\}$&$(1\;2\;3)(4\;5\;6)(7\;8)(9)$&$(2\;1\;7)(6\;5\;9)(3\;4)(8)$&$(2\;4\;9\;6\;3\;7\;8)(1)(5)$\\
\hline
9&$4$&$9$&$\{[3,3,2,1],[3,2,2,2]\}$&$(1\;2\;3)(4\;5\;6)(7\;8)(9)$&$(2\;1\;9)(4\;7)(6\;5)(3\;8)$&$(2\;8\;4\;6\;7\;3\;9)(1)(5)$\\
\hline
10&$4$&$9$&$\{[3,2,2,2,],[3,2,2,2]\}$&$(1\;2\;3)(4\;5)(6\;7)(8\;9)$&$(5\;4\;6)(3\;2)(1\;8)(7\;9)$&$(1\;3\;8\;7\;5\;6\;9)(2)(4)$\\
\hline
11&$4$&$11$&$\{[3,3,3,2],[3,3,2,1,1,1]\}$&$(1\; 2\; 3)(4\;5\;6)(7\;8\;9)(10\;11)$&$(5\;4\;1)(9\;8\;6)(2\;10)(3)(11)(7)$&$(1\;10\;11\;2\;3\;5\;9\;7\;6)(4)(8)$\\
\hline
12&$4$&$11$&$\{[3,3,3,2],[3,2,2,2,1,1]\}$&$(1\; 2\; 3)(4\;5\;6)(7\;8\;9)(10\;11)$&$(3\;4\;7)(2\;1)(6\;5)(9\;10)(8)(11)$&$(2\;4\;6\;7\;8\;10\;11\;9\;3)(1)(5)$\\
\hline
13&$4$&$11$&$\{[3,3,3,2],[3,3,3,2]\}$&$(1\; 2\; 3)(4\;5\;6)(7\;8\;9)(10\;11)$&$(3\;2\;4)(6\;5\;7)(8\;9\;10)(1\;11)$&$(1\;4\;7\;9\;6\;3\;11\;8\;10)(2)(5)$\\
\hline
14&$5$&$11$&$\{[3,2,2,2,2],[3,2,2,2,2]\}$&$(1\; 2\; 3)(4\;5)(6\;7)(8\;9)(10\;11)$&$(5\;4\;6)(2\;3)(1\;8)(9\;10)(7\;11)$&$(1\;3\;8\;10\;7\;5\;6\;11\;9)(2)(4)$\\
\hline
15&$5$&$11$&$\{[3,2,2,2,2],[3,3,3,1,1]\}$&$(1\; 2\; 3)(4\;5)(6\;7)(8\;9)(10\;11)$&$(5\;4\;6)(9\;8\;10)(1\;7\;11)(2)(3)$&$(1\;2\;3\;7\;5\;6\;11\;9\;10)(4)(8)$\\
\hline
16&$5$&$11$&$\{[3,2,2,2,2],[3,3,2,2,1]\}$&$(1\; 2\; 3)(4\;5)(6\;7)(8\;9)(10\;11)$&$(5\;4\;6)(9\;8\;10)(11\;1)(2\;7)(3)$&$(1\;7\;5\;6\;2\;3\;11\;9\;10)(4)(8)$\\
\hline
17&$5$&$11$&$\{[3,3,3,1,1],[3,3,3,1,1]\}$&$(1\; 2\; 3)(4\;5\;6)(7\;8\;9)(10)(11)$&$(3\;2\;4)(6\;5\;7)(9\;10\;11)(1)(8)$&$(1\;4\;7\;8\;10\;11\;9\;6\;3)(2)(5)$\\
\hline
18&$5$&$11$&$\{[3,3,3,1,1],[3,3,2,2,1]\}$&$(1\; 2\; 3)(4\;5\;6)(7\;8\;9)(10)(11)$&$(3\;2\;4)(6\;5\;7)(9\;10)(1\;11)(8)$&$(1\;4\;7\;8\;10\;9\;6\;3\;11)(2)(5)$\\
\hline
19&$5$&$11$&$\{[3,3,2,2,1],[3,3,2,2,1]\}$&$(1\; 2\; 3)(4\;5\;6)(7\;8)(9\;10)(11)$&$(3\;2\;4)(8\;7\;6)(9\;11)(1\;10)(5)$&$(1\;4\;5\;8\;6\;3\;10\;11\;9)(2)(7)$\\
\hline
\end{tabular}
\end{scriptsize}

Departament of Mathematics

University Federal  of S\~ao Carlos

Rod. Washington Luis, Km. 235. C.P 676 - 13565-905 S\~ao Carlos, SP - Brazil

nbedoya@dm.ufscar.br

\vspace{0.25cm}

Departament of Mathematics 

Institute of Mathematics and Statistic

University of  S\~ao Paulo

Rua do Mat\~ao 1010, CEP 05508-090, S\~ao Paulo, SP, Brazil.

dlgoncal@ime.usp.br

\vspace{0.25cm}

Department of Mathematics and Mechanics,

Moscow State University

Moscow 119992, Russia

ekudr@gmx.de


\begin{thebibliography}{99}


\bibitem {BN0} N. A.  V. Bedoya, Branched coverings over compact surfaces(Portuguese, with English summary),
Phd. thesis IME-University of S\~ao Paulo-Brazil June-2008.

\bibitem {BN&GDL1} N. A.  V. Bedoya and D. L. Gon\c{c}alves, {\it The problem of the decomposability of  branched coverings},  Mat. Sb, \textbf{201} 12 (2010),  3-20.

\bibitem {BN&GDL2} N. A.  V. Bedoya and D. L. Gon\c{c}alves, {\it
Primitivity of monodromy groups of branched coverings: a non-orientable case},
JP Journal of Geometry and Topology 12:2 (2012), 219-234.




\bibitem {BBZ} S. Bogataya, S. Bogaty{\u\i} and H. Zieschang, {\it  On compositions
   of open mappings}, {\it Mat. Sb.}, \textbf{193}, 3-20 (2002)


\bibitem {BGKZ1} S. A.  Bogatyi, D. L. Gon{\c{c}}alves, E. A. Kudryavtseva and
              H. Zieschang, {\it Realization of primitive branched coverings over
              closed  surfaces}, Advances in topological quantum field
              theory, \textbf{179}, 297-316 (2004)

\bibitem {BM} K.  Borsuk and R. Molski, {\it On a class of continuous mappings},
Fund. Math., \textbf{45}, 84-98 (1957)


\bibitem {DM} J. D.Dixon  and B. Mortimer, {\it Permutation groups},  Graduate
Texts in Mathematics, \textbf{163}, (1996)

\bibitem {EKS} A. L. Edmonds, R. S. Kulkarni and R. E. Stong, {\it Realizability of
  branched coverings of surfaces}, Trans. Amer. Math. Soc., \textbf{282},
 (1984)

\bibitem {Ez} C. L. Ezell, {\it Branch point structure of covering maps onto
nonorientable surfaces},  Trans. Amer. Math. Soc.T, \textbf{243}, 123-133 (1978)

\bibitem {Hurwitz} A.  Hurwitz, {\it \"Uber {R}iemaniannische      {F}l\"achen mit gegebenen
              {V}erzweigungspunkten},  Math. Ann., \textbf{39}, 1-60 (1891)

\bibitem {Hu} D. H. Husemoller, {\it Ramified coverings of {R}iemann surfaces},
Duke Math. J., \textbf{29}, 167-174 (1962)

\bibitem {Thom} R.\ Thom, L'equivalence d'une fonction diff\'erentiable et d'un polynome, Topology (2) {\bf 3} (1965), 297-307.
\end{thebibliography}
\end{document}